\documentclass[11pt]{amsart}
\usepackage{xy}
\xyoption{all}
\usepackage[pdftitle={An Easy Proof of the Stone-von Neumann-Mackey Theorem}, pdfauthor={Amritanshu Prasad}, pdfkeywords={Stone von Neumann, Heisenberg group, Mackey, Peter-Weyl, Fourier transform, locally compact abelian group}, pdfdisplaydoctitle={true}]{hyperref}
\DeclareMathSizes{12}{10}{8}{6}

\numberwithin{equation}{section}
\theoremstyle{plain}
\newtheorem{theorem}{Theorem}[section]
\newtheorem*{theorem*}{Theorem}

\newtheorem*{prop*}{Proposition}
\newtheorem{lemma}[theorem]{Lemma}
\newtheorem*{lemma*}{Lemma}

\newtheorem*{cor*}{Corollary}

\theoremstyle{definition}

\newtheorem*{defn*}{Definition}

\newtheorem*{remark*}{Remark}

\newcommand{\Hilb}{\mathcal H}

\newcommand{\N}{\mathbf N}
\newcommand{\Q}{\mathbf Q}
\newcommand{\R}{\mathbf R}

\newcommand{\ses}[5]{
  \xymatrix{
    0 \ar[r] & #1 \ar[r]^{#4} & #2 \ar[r]^{#5} & #3 \ar[r] & 0
  }
}

\newcommand{\Z}{\mathbf Z}

\title[Stone-von Neumann-Mackey Theorem]{An Easy Proof of the\\Stone-von Neumann-Mackey Theorem}
\author{Amritanshu Prasad}
\address{\href{http://www.imsc.res.in/}{The Institute of Mathematical Sciences}\\Chennai.}
\urladdr{http://www.imsc.res.in/~amri}
\email{amri@imsc.res.in}
\begin{document}
\begin{abstract}
  The Stone-von Neumann-Mackey Theorem for Heisenberg groups associated to locally compact abelian groups is proved using the Peter-Weyl theorem and the theory of Fourier transforms for $\R^n$.
  A theorem of Pontryagin and van Kampen on the structure of locally compact abelian groups (which is evident in any particular case) is assumed.
\end{abstract}
\maketitle
\section{Introduction}
The definition of a Heisenberg group, which is motivated by Heisenberg's commutation relations for position and momentum operators, goes back to Hermann Weyl's mathematical formulation of quantum kinematics \cite{Weyl50}.
The basic feature of Heisenberg groups, now known as the Stone-von Neumann theorem, was proved for Heisenberg groups associated to real vector spaces by Marshall Stone \cite{Stone30} and John von Neumann \cite{vN31}.
George W. Mackey \cite{Mackey49} extended this result to Heisenberg groups associated to all locally compact abelian groups, allowing for groups that arise in number theory with consequences that were first studied by Andr\'e Weil \cite{MR0165033}.

Our proof of Mackey's version of the Stone-von Neumann theorem uses relatively simple tools from functional analysis: the Peter-Weyl Theorem on representations of compact topological groups \cite[Theorem~1.12]{MR855239}, and the classical theory of Fourier transforms for real vector spaces, specifically, the Fourier Inversion Theorem and the Plancherel Theorem \cite[Chap.~9]{MR924157}.
Although we assume a theorem of Pontryagin and van Kampen on the structure of locally compact abelian groups (see \S~\ref{sec:locally-comp-abel}), the conclusion of this theorem is obvious for groups that typically come up in applications.

Sections~\ref{sec:locally-comp-abel} and \ref{sec:example-locally-comp} are introductory in nature.
Since they are only meant to convey certain general ideas, some details are omitted.
In \S~\ref{sec:stone-von-neumann}, after stating the Stone-von Neumann-Mackey theorem, we give a new proof (based on the Peter-Weyl Theorem) for locally compact abelian groups with compact open subgroups.
We then recall von Neumann's original proof in the real case (which uses the theory of Fourier transforms for real vector spaces).
Finally, we explain how to deduce the theorem for groups which are the product of an abelian group with a compact open subgroup and a real vector space.
\section{Locally Compact Abelian Groups, Pontryagin Duality}
\label{sec:locally-comp-abel}
Let $L$ be a locally compact (Hausdorff) abelian group.
Its Pontryagin dual is the group $\hat L$ of continuous homomorphisms $L\to T$, where $T$ denotes the circle group $\R/\Z$.
When equipped with the compact open topology, $\hat L$ is also a locally compact abelian group.

Two fundamental properties of such groups are closely interrelated:
\begin{theorem*}
  [Pontryagin-van~Kampen Fundamental Structure Theorem]
  Every locally compact abelian group is isomorphic to $E\times \R^n$ for some locally compact abelian group $E$ which has a compact open subgroup and a positive integer $n$.
\end{theorem*}
\begin{theorem*}
  [Pontryagin Duality Theorem]
  The map $\phi:L\to \hat{\hat L}$ defined by $\phi(x)(\chi)=\chi(x)$ is an isomorphism of topological groups.
\end{theorem*}
Indeed, one can deduce the Pontryagin Duality Theorem for groups which have a compact open subgroup from the Peter-Weyl theorem (see \S~\ref{sec:groups-with-compact} for a further hint on this), and combining this with the duality theorem for $\R^n$ (which is elementary), one may deduce the Pontryagin Duality Theorem for all locally compact abelian groups from the Fundamental Structure Theorem.
The relationship in the other direction is more subtle, and the reader is referred to \cite{MR0442141}.
\section{Examples of Locally Compact Abelian Groups}
\label{sec:example-locally-comp}
\subsection{The groups $\R^n$}
\label{sec:reals}
This is the classical example where the Stone-von Neumann theorem was first proved \cite{Stone30,vN31}.
For each $y\in \R^n$, define $\chi_y:\R^n\to T$ by 
\begin{equation*}
  \chi_y(x)=x y+\Z
\end{equation*}
where $x y$ denotes the scalar product of $x$ and $y$.
Then $y\mapsto\chi_y$ is an isomorphism $\R\to \hat \R$.
\subsection{Discrete Abelian Groups}
\label{sec:discrete}
The dual of a discrete abelian group is compact.
Indeed, if $A$ is discrete, the topology on $\hat A$ is the that of point-wise convergence.
Thus $\hat A$ can be thought of as a closed subgroup of $T^A$ (with the product topology) from which it inherits a compact topology.

An interesting example is the Pr\"ufer $p$-group $\Z(p^\infty)$, which (for each prime $p$) is defined as the union in $T$ of the subgroups
\begin{equation*}
  T[p^n]=\{x\in T: p^n x=0\},
\end{equation*}
and is given the discrete topology (with the subspace topology coming from $T$, $\Z(p^\infty)$ would not be locally compact).
A character is completely determined by the images of $p^{-n}$, $n\in \N$.
If for each $n$, $p^{-n}\mapsto a_n p^{-n}$, then we may think of $a_n$ as an element of $\Z/p^n\Z$.
Thus, the Pontryagin dual of the Pr\"ufer group is the subgroup of $\prod_{n\in \N}\Z/p^n\Z$ consisting on sequences $(a_n)$ such that the image of $a_n$ in $\Z/p^{n-1}\Z$ is $a_{n-1}$.
In addition to being added, the sequences $a_n$ can also be multiplied element-wise.
The result is the ring $\Z_p$ of $p$-adic integers.
\subsection{Compact Abelian Groups}
\label{sec:compact}
The dual of a compact abelian group is discrete.
To see this, firstly observe that if $G$ is a compact abelian group, then the image of any non-zero continuous character $\chi:G\to T$ is a non-trivial compact subgroup of $T$, therefore it is either all of $T$, or the cyclic subgroup of $n$-torsion points for some positive integer $n$.
In any case, it can never be contained in the image of $(-1/4,1/4)\subset \R$ in $T$.
Since the topology on $\hat G$ is that of uniform convergence, it follows that $0$ is an isolated point in $\hat G$, so $\hat G$ is discrete.

It is a consequence of the Peter-Weyl Theorem that Pontryagin duality gives rise to a contravariant isomorphism between the category of discrete abelian groups and the category of compact abelian groups.
\subsection{Groups with compact open subgroups}
\label{sec:groups-with-compact}
If a locally compact abelian group $E$ has a compact open subgroup $G$, then the quotient $A=E/G$ is discrete.
Let $G^\perp$ be the subgroup of $\hat E$ consisting of homomorphisms $E\to T$ which vanish on $G$.
Then $G^\perp$ is isomorphic to the Pontryagin dual of $A$ and is therefore compact.
On the other hand, the quotient $\hat E/G^\perp$ is identified with $\hat G$, and is therefore a discrete abelian group.
Thus $G^\perp$ is a compact open subgroup of $\hat E$.
The Duality Theorem for groups with compact open subgroups is a consequence of the duality between discrete and compact groups.
A basic example here is the field $\Q_p$ of $p$-adic numbers, which is the field of fractions of $\Z_p$.
It is topologized in such a way that $\Z_p$ is a compact open subgroup.
The discrete quotient $\Q_p/\Z_p$ is isomorphic to the Pr\"ufer $p$-group $\Z(p^\infty)$.
\subsection{The ad\`eles of $\Q$}
\label{sec:adeles}
Recall that the ring of ad\`eles of $\Q$ is $\mathbf A=\mathbf A_f \times \R$, where $\mathbf A_f$ is the ring of finite ad\`eles.
The factor $\mathbf A_f$ has a compact open subgroup isomorphic to $G=\prod_p \Z_p$ (a product over all primes) such that the quotient $\mathbf A_f/G$ is isomorphic to $\oplus_p \Z(p^\infty)$.
\section{Stone-von Neumann-Mackey Theorem}
\label{sec:stone-von-neumann}
\subsection{Statement of the Theorem}
\label{sec:statement-theorem}
Let $L$ be a locally compact abelian group.
Consider the following unitary operators on $L^2(L)$:
\begin{tabbing}
  Translation Operators:\hspace{1cm}\=$T_x f(u)=f(u-x)$ for $x\in L$,\\
  Modulation Operators: \>$M_\chi f(u)=e^{2\pi i\chi(u)}f(u)$ for $\chi \in \hat L$.
\end{tabbing}
These operators satisfy the commutation relation
\begin{equation}
  \label{eq:commutation}
  [T_x,M_\chi]=e^{-2\pi i \chi(x)}\mathrm{Id}_{L^2(L)}
\end{equation}
It follows that
\begin{equation*}
  H:=\big\{ e^{2\pi it}T_x M_\chi:t\in T,\; x\in L,\; \chi\in \hat L\big\}
\end{equation*}
is a subgroup of the group of unitary operators on $L^2(L)$. This group $H$ is called the Heisenberg group associated to $L$.
The resulting representation of $H$ on $L^2(L)$ is called the canonical representation of $H$.
\begin{theorem*}
  [Stone-von Neumann-Mackey]
  \begin{enumerate}
  \item The Hilbert space $L^2(L)$ has no non-trivial proper closed subspace invariant under $H$.
  \item If $\Hilb$ is a Hilbert space and $\rho:H\to U(\Hilb)$ is a continuous homomorphism\footnote{$U(\Hilb)$ is given the strong operator topology.} such that $\rho(e^{2\pi it})=e^{2\pi it}\mathrm{Id}_\Hilb$, there exists an orthogonal sum decomposition of Hilbert spaces
    \begin{equation*}
      \Hilb=\bigoplus_{\alpha} \Hilb^\alpha
    \end{equation*}
    such that for each $\alpha$ there exists, up to scaling, a unique isometry $W_\alpha:L^2(L)\to \Hilb^\alpha$ satisfying $W_\alpha(g f)= \rho(g)W_\alpha f$ for all $g\in H$ and $f\in L^2(L)$.
  \end{enumerate}
\end{theorem*}
\subsection{The proof for groups with compact open subgroups}
\label{sec:proof-groups-with}
Let $G$ be a compact open subgroup of $E$ and $A=E/G$.
As explained in \S~\ref{sec:groups-with-compact}, $G^\perp$ is a compact open subgroup of $\hat E$ which is isomorphic to $\hat A$.
We have short exact sequences
\begin{gather*}
  \ses G E A j q\\
  \ses {\hat A}{\hat E}{\hat G}{\hat q}{\hat j}
\end{gather*}
Let $\rho$ be as in the second part of the Stone-von Neumann-Mackey Theorem.
The commutation relation (\ref{eq:commutation}) implies that the operators $T_g$ and $M_\xi$ commute for all $g\in G$ and $\xi\in G^\perp$.
Thus $\rho$ is a continuous unitary representation of the abelian group $G\times G^\perp$ on $\Hilb$.
By the Peter-Weyl theorem
\begin{equation*}
  \Hilb=\oplus_{(\eta,a) \in \hat G\times A} \Hilb_{\eta,a},
\end{equation*}
where
\begin{equation*}
  \Hilb_{\eta,a}= \big\{v\in \Hilb: \rho(T_g M_\xi) v = e^{2\pi i(\eta(g)+\xi(a))} v\text{ for all } g\in G, \xi \in G^\perp\big\}.
\end{equation*}
The hypothesis that $\rho(e^{2\pi i t})=e^{2\pi i t}\mathrm{Id}_\Hilb$ implies that the commutation relation (\ref{eq:commutation}) is inherited by the operators $\rho(T_x)$ and $\rho(M_\chi)$.
Therefore, for any $(x,\chi)\in E\times \hat E$,
\begin{eqnarray*}
  \rho(T_g M_\xi) \rho(T_x M_\chi) & = & e^{2\pi i \xi(x)} \rho(T_g) \rho(T_x) \rho(M_\xi) \rho(M_\chi)\\
  & = & e^{2\pi i \xi(x)}\rho(T_x) \rho(T_g) \rho(M_\chi) \rho(M_\xi)\\
  & = & e^{2\pi i(\xi(x)-\chi(g))}\rho(T_x M_\chi) \rho(T_g M_\xi).
\end{eqnarray*}
Strictly speaking, $\chi(g)=\chi(j(g))$, which is the same as $\hat j(\chi)(g)$.
It follows that if $v\in \Hilb_{\eta,a}$, then $\rho(T_x M_\chi) v\in \Hilb_{\eta-\hat j(\chi), a+q(x)}$.
Thus $H$ permutes the subspaces $\Hilb_{\eta,a}$ transitively.
In particular, they are all of the same dimension, and unless $\Hilb=0$, each $\Hilb_{\eta,a}\neq 0$.

Now apply the above considerations to the canonical representation of $H$ on $L^2(E)$.
Suppose $f\in L^2(E)_{0,0}$.
Then $f$ is invariant under translations in $G$.
We claim that $f$ vanishes almost everywhere outside $G$.
Indeed, if $x\notin G$, there exists $\xi\in G^\perp$ such that $\xi(x)\neq 0$.
By continuity, $\xi$ does not vanish anywhere in an open neighborhood of $x$.
Since $M_\xi f(x)=e^{2\pi i \xi(x)}f(x)$ coincides with $f(x)$ almost everywhere, $f$ must vanish almost everywhere in this neighborhood.
By varying $x$, one concludes that $f$ vanishes almost everywhere outside $G$, as claimed.
Therefore $f$ must be a scalar multiple of the characteristic function $\delta_G$ of $G$.
It follows that $L^2(E)_{\eta,a}$ is one-dimensional for all $(\eta,a)\in \hat G\times A$.

Let $M$ be a non-trivial subspace of $L^2(E)$ which is invariant under $H$.
The Peter-Weyl theorem applied to the representation of $G\times G^\perp$ on $M$ implies that $M=\oplus_{\eta,a} M_{\eta,a}$, where $M_{\eta,a}=M\cap L^2(E)_{\eta,a}$.
Thus $M_{\eta,a}\neq 0$ for at least one, and hence all $(\eta,a)\in \hat G\times A$.
Since each $L^2(E)_{\eta,a}$ is itself one dimensional, we have $M=L^2(E)$, proving the first part of the theorem.

To prove the second part, let $\{v_\alpha\}$ be an orthonormal basis of $\Hilb_{0,0}$ and let $\Hilb^{\alpha}$ be the closure of the subspace of $\Hilb$ spanned by $\rho(g)v_\alpha$ as $g$ varies over $H$.
Clearly, each $\Hilb^\alpha$ is an invariant subspace of $\Hilb$.
The orthogonal complement $N$ of $\oplus_\alpha \Hilb^\alpha$ is also an invariant subspace of $\Hilb$.
Since $N_{0,0}=N\cap \Hilb_{0,0}$ is trivial, $N$ must also be trivial.
Thus $\Hilb=\oplus \Hilb^\alpha$.

Now $W_\alpha:L^2(E)\to \Hilb^\alpha$, if it exists, must take $L^2(E)_{0,0}$ to $\Hilb^\alpha_{0,0}$.
After scaling, we may assume it takes $\delta_G$ to $v_\alpha$, and therefore $T_x M_\chi \delta_G$ to $\rho(T_x M_\chi)v_\alpha$.
When $(x,\chi)$ runs over a set of coset representatives of $G\times G^\perp$ in $E\times \hat E$, $T_x M_\chi \delta_G$ and $\rho(T_x M_\chi) v_\alpha$ run over orthonormal bases of $L^2(E)$ and $\Hilb^\alpha$ respectively (these vectors are pairwise orthogonal because they belong to distinct eigenspaces for the action of $G\times G^\perp$).
Therefore, the function taking $T_x M_\chi \delta_G$ to $\rho(T_x M_\chi) v_\alpha$ extends to an isometry of Hilbert spaces from $L^2(E)$ onto $\Hilb^\alpha$.
This completes the proof of the Stone-von Neumann-Mackey Theorem when $E$ has a compact open subgroup.
\subsection{The proof for $\R^n$}
\label{sec:proof-Rn}
We follow Folland's exposition \cite{MR983366} of von Neumann's original proof.
Identify $\R^n$ with its Pontryagin dual using the map $y\mapsto \chi_y$ as in \S~\ref{sec:reals} and write $M_y$ for the modulation operator $M_{\chi_y}$.
\subsubsection*{The Weyl operators}
For $k=(x,y)\in \R^{2n}$, let $W_k=e^{\pi i x y}T_x M_y$.
These unitary operators (called Weyl operators) satisfy the identities
\begin{equation*}
  W_k^*=W_{-k},\quad W_kW_l=e^{\pi i \omega(k,l)}W_{k+1} \text{ for all } k,l\in \R^{2n}.
\end{equation*}
Here $\omega(k,l)$ is the standard symplectic form on $\R^{2n}$; if $k=(x,y)$ and $l=(u,v)$ ($x,y,u,v\in \R^n$), then
\begin{equation*}
  \omega(k,l)=y u-xv.
\end{equation*}
\subsubsection*{The Fourier-Wigner transform}
Given $f,g\in L^2(\R^n)$, their Fourier-Wigner transform is defined as
\begin{equation*}
  V(f,g)(k)=\langle W_k f,g\rangle.
\end{equation*}
The Cauchy-Schwartz inequality implies that $V(f,g)$ is bounded above by $\|f\|\|g\|$, and the strong continuity of $k\mapsto W_k$ implies that $V(f,g)$ is a continuous function of $k$.
We calculate
\begin{eqnarray*}
  V(f,g)(k)&=&\int_{\R^n} e^{2\pi i y(u-x/2)}f(u-x)\overline{g(u)}d u \\
  &=&\int_{\R^n}e^{2\pi i y u}f(u-x/2)\overline{g(u+x/2)} d u.
\end{eqnarray*}
Taking $F(x,y)=f(x)\overline{g(y)}$, we see that $V(f,g)$ is a Fourier transform in one variable of a function obtained from $F$ by the volume-preserving change of variables $(u,x)\mapsto (u-x/2,u+x/2)$, and therefore extends to an isometry of $L^2(\R^{2n})$ onto $L^2(\R^{2n})$.
Therefore,
\begin{equation}
  \label{eq:1}
  \langle V(f,g),V(\phi,\psi)\rangle = \langle f\otimes \overline g, \phi\otimes \overline \psi \rangle = \langle f,\phi\rangle \langle \overline g,\overline \psi\rangle
\end{equation}
\subsubsection*{Proof of the first part}
We can now deduce the first part of the Stone-von Neumann Theorem as follows: let $M$ be a non-trivial subspace of $L^2(\R^{2n})$ which is invariant under the Heisenberg group.
Let $f$ be a non-zero vector in $M$.
If $g\perp M$, then $g\perp W_k f$ for all $k\in \R^{2n}$.
Consequently, $V(f,g)\equiv 0$.
By (\ref{eq:1}), $\|f\|\|g\|=\|V(f,g)\|=0$, so $g=0$ in $L^2(\R^{2n})$.
It follows that $M=L^2(\R^{2n})$.
\subsubsection*{Idea behind the proof of the second part}
The orthogonal projection onto the line spanned by the Gaussian in $L^2(\R^n)$ is identified with an operator $W_\Phi$ constructed from the action of the Heisenberg group on $L^2(\R^n)$.
The corresponding operator $\rho(W_\Phi)$ on $\Hilb$ is also a projection operator, and its range $N$ is a space whose unit vectors the Gaussian can be mapped to (this should be seen as analogous to $\Hilb_{0,0}$ being the space to whose unit vectors $\delta_G$ can be mapped to).
The remaining steps correspond to those in \S~\ref{sec:proof-groups-with}.
\subsubsection*{Proof of the second part}
For $\Phi\in L^1(\R^{2n})$, the formula
\begin{equation*}
  \langle \rho(W_\Phi)v,w\rangle =\int_{\R^{2n}} \Phi(k)\langle \rho(W_k) v,w\rangle d k
\end{equation*}
gives rise to a well-defined bounded operator $\rho(W_\Phi):\Hilb\to \Hilb$.
We abbreviate the above definition and write
\begin{equation}
  \label{eq:3}
  \rho(W_\Phi)=\int \Phi(k)\rho(W_k) d k.
\end{equation}
One should think of $\rho(W_\Phi)$ as an averaging-out of the action of $\rho$ using $\Phi$ as a density.
Formally, when $\Phi$ is the delta function at $k$, $\rho(W_\Phi)$ is just $\rho(W_k)$.
A formal calculation yields
\begin{eqnarray*}
  \rho(W_\Phi)\circ \rho(W_\Psi) & = & \int \Phi(k)\rho(W_k) \int \Psi(l)\rho(W_l) d l d k\\
  & = & \int \int \Phi(k)\Psi(l)e^{\pi i \omega(k,l)}\rho(W_{k+l}) d l d k\\
  & = & \int \int \Phi(k-l)\Psi(l)e^{\pi i \omega(k-l,l)}d l \rho(W_k)d k.
\end{eqnarray*}
The twisted convolution
\begin{equation*}
  \Phi\#\Psi(k)=\int e^{\pi i \omega(k-l,l)}\Phi(k-l)\Psi(l)d l
\end{equation*}
makes sense for $\Phi$ and $\Psi$ in $L^1(\R^{2n})$ and $\|\Phi\#\Psi\|_1\leq \|\Phi\|_1\|\Psi\|_1$.
With this in mind, the identity $\rho(W_\Phi)\circ \rho(W_\Psi)=\rho(W_{\Phi\#\Psi})$ suggested by the above formal calculation can be justified using Fubini's theorem.
Also, it is not difficult to see that $\rho(W_{\Phi})^*=\rho(W_{\Phi^*})$, where $\Phi^*(k)=\overline{\Phi(-k)}$.
\begin{lemma}
  \label{lemma:faithful}
  For $\Hilb\neq 0$ and $\Phi\in L^1(\R^{2n})$, if  $\rho(W_\Phi)=0$ then $\Phi=0$ almost everywhere. 
\end{lemma}
\begin{proof}
  For any $u,v\in \Hilb$,
  \begin{eqnarray*}
    0 & = & \langle \rho(W_k)\rho(W_\Phi)\rho(W_{-k})u,v\rangle\\
    & = & \int \Phi(k)\langle \rho(W_k)\rho(W_l)\rho(W_{-k})u,v\rangle d l\\
    & = & \int \Phi(k)e^{2\pi i \omega(k,l)} \langle \rho(W_l)u,v\rangle d l\\
    & = & \hat \Psi(\tilde k),
  \end{eqnarray*}
  where $\tilde k\in \R^{2n}$ is such that $\omega(k,l)=\tilde k l$ for all $l\in \R^{2n}$ and $\Psi(l)=\Phi(l)\langle \rho(W_l)u,v\rangle$.
  It follows from the Fourier inversion theorem that $\Psi(l)=0$ almost everywhere.
  By taking $u$ to be a unit vector and $v=\rho(W_k)u$, one sees that $\Phi(l)=0$ almost everywhere in a neighborhood of $k$.
  By varying $k$, it follows that $\Phi(l)=0$ almost everywhere.
\end{proof}

When $\rho$ is the canonical representation, we denote $\rho(W_\Phi)$ simply by $W_\Phi$.
\begin{lemma}
  \label{lemma:projection}
  If $\Phi=\overline{V(\phi,\psi)}$ and $\Phi\in L^1(\R^{2n})$ then $W_\Phi f = \langle f,\phi\rangle \psi$.
\end{lemma}
\begin{proof}
  Indeed, using (\ref{eq:1}), for any $g\in L^2(\R^n)$,
  \begin{eqnarray*}
    \langle W_\Phi f, g\rangle & = &  \int \overline{V(\phi,\psi)(k)} \langle W_k f, g\rangle\\
    & = & \langle V(f,g),V(\phi,\psi)\rangle\\
    & = & \langle f, \phi \rangle \langle \psi, g\rangle
  \end{eqnarray*}
  from which the result follows.
\end{proof}
Thus if $\psi$ is a unit vector and $\phi=\psi$, then $W_\Phi$ is the orthogonal projection onto the line spanned by $\psi$.
In particular, it is idempotent and self-adjoint.
Consequently, $W_{\Phi\#\Phi}=W_\Phi\circ W_\Phi=W_\Phi$.
By Lemma~\ref{lemma:faithful}, it follows that $\Phi\#\Phi=\Phi$.
Similarly, since $W_\Phi^*=W_\Phi$, $\Phi^*=\Phi$.
Therefore, for any $\rho$ as in the second part of the Stone-von Neumann-Mackey Theorem, $\rho(W_\Phi)^2=\rho(W_{\Phi\#\Phi})=\rho(W_\Phi)$ and $\rho(W_\Phi^*)=\rho(W_{\Phi^*})=\rho(W_\Phi)$.
Therefore $\rho(W_\Phi)$ is an orthogonal projection in $\Hilb$.

For example, take both $\psi$ and $\phi$ to be the normalised Gaussian:
\begin{equation*}
  \phi(u)=2^{n/4}e^{-\pi u^2}
\end{equation*}
then writing $k=(x,y)$, we have
\begin{eqnarray*}
  \Phi(k) & = & \langle W_k \phi,\phi \rangle\\
  & = & 2^{n/2} \int e^{2\pi i y(u-x/2)-\pi (u-x/2)^2-\pi u^2}d u\\
  & = & e^{-\pi x^2/2}\int e^{-\pi v^2+2\pi i2^{-1/2}y v}d v
\end{eqnarray*}
where $v=2^{1/2}u-2^{-1/2}x$.
This integral is precisely the Fourier transform of the Gaussian evaluated at $2^{-1/2}y$.
Since the Gaussian is its own Fourier transform, we have
\begin{equation*}
  \Phi(k)=e^{-\pi(x^2+y^2)/2}.
\end{equation*}
In particular, $\Phi(k)\in L^1(\R^{2n})$.

Let $N$ be the range of the projection operator $\rho(W_\Phi)$.
Let $\{v_\alpha\}$ be an orthonormal basis of $N$.
Let $H^\alpha$ be the closed subspace of $\Hilb$ spanned by $\rho(W_k)v_\alpha$ as $k$ varies over $\R^{2n}$.
Clearly, each $\Hilb^\alpha$ is an invariant subspace for $\Hilb$.
Furthermore, for $u,v\in N$, since $\rho(W_\Phi)u=u$ and $\rho(W_\Phi)v=v$,
\begin{eqnarray*}
  \langle \rho(W_k)u, \rho(W_l)v\rangle & = & \langle \rho(W_k)\rho(W_\Phi)u, \rho(W_l)\rho(W_\Phi)v\rangle\\
  & = & e^{\pi i\omega(k,l)}\langle \rho(W_\Phi)\rho(W_{k+l})\rho(W_\Phi)u,v\rangle.
\end{eqnarray*}

To proceed further, we need to evaluate an expression which is of the form $\rho(W_\Phi)\rho(W_k)\rho(W_\Phi)$.
Observe that
\begin{eqnarray*}
  \rho(W_k)\rho(W_\Phi) & = & \int \Phi(l) \rho(W_k) \rho(W_l) d l\\
  & = & \int \Phi(l)e^{\pi i \omega(k,l)}\rho(W_{k+l})d l\\
  & = & \int \Phi(l-k)e^{\pi i\omega(k,l-k)}\rho(W_l)d l\\
  & = & \rho(W_{\Phi^k}),
\end{eqnarray*}
where $\Phi^k(l)=e^{\pi i\omega(k,l-k)}\Phi(l-k)$.
In particular, for the canonical representation, we have that $W_kW_\Phi=W_{\Phi^k}$.
On the other hand,
using Lemma~\ref{lemma:projection}, we find that
\begin{eqnarray*}
  W_\Phi W_k W_\Phi f & = & \langle f, \phi \rangle \langle W_k \phi,\phi\rangle \phi\\
  & = &  W_{\langle W_k\phi,\phi\rangle \Phi}f.
\end{eqnarray*}
It follows that $\Phi\#\Phi^k=\langle W_k\phi,\phi\rangle\Phi$.
Therefore
\begin{equation*}
  \rho(W_\Phi)\rho(W_k)\rho(W_\Phi)=\rho(W_{\Phi\#\Phi^k})=\langle W_k\phi,\phi\rangle\rho(W_\Phi).
\end{equation*}
Returning to the calculation of $\langle \rho(W_k)u,\rho(W_l)v\rangle$, we may conclude that
\begin{equation}
\label{eq:2}
  \langle \rho(W_k)u,\rho(W_l)v\rangle = e^{\pi i \omega(k,l)}\langle W_k\phi,\phi\rangle\langle u,v\rangle 
\end{equation}
If $\alpha\neq \beta$, then $v_\alpha \perp v_\beta$, and therefore by (\ref{eq:2}) $\Hilb^\alpha \perp \Hilb^\beta$.

Thus $\oplus_\alpha \Hilb^\alpha$ is an invariant orthogonal sum of Hilbert spaces.
Its orthogonal complement $M=(\oplus_\alpha \Hilb^\alpha)^\perp$ is also invariant.
Therefore, $\rho(W_\Phi)M \subset N\cap M$, which is trivial.
By Lemma~\ref{lemma:faithful}, it follows that $M=0$.

We have no choice but to define $W_\alpha$ by
\begin{equation*}
  W_\alpha(W_k \phi)=\rho(W_k)v_\alpha.
\end{equation*}
By (\ref{eq:2}),
\begin{equation*}
  \langle W_k\phi, W_l\phi\rangle = e^{\pi i \omega(k,l)}\langle W_k\phi,\phi\rangle = \langle \rho(W_k)v_\alpha,\rho(W_l)v_\alpha\rangle.
\end{equation*}
Therefore, $W_\alpha$ does indeed extend to an isometry of $L^2(\R^n)$ onto $H^\alpha$.
\subsection{The general case}
By the Fundamental Structure Theorem, we have $L=E\times \R^n$ for some locally compact abelian group $E$ with a compact open subgroup $G$ and some non-negative integer $n$.
The Heisenberg group $H$ of $L$ contains the Heisenberg groups $H_E$ and $H_{\R^n}$ of $E$ and $\R^n$ respectively (these subgroups commute with each other and share the centre).
Moreover, $L^2(L)=L^2(E)\otimes L^2(\R^n)$ (a Hilbert space tensor product).
Recall from \S~\ref{sec:proof-groups-with} that $L^2(E)$ is a sum of one-dimensional eigenspaces with distinct eigencharacters:
\begin{equation*}
  L^2(E)=\oplus_{(\eta,a)\in \hat G\times A} L^2(E)_{\eta,a},
\end{equation*}
Therefore,
\begin{equation*}
  L^2(L)=\oplus_{\eta,a}\big( L^2(E)_{\eta,a}\otimes L^2(\R^n)\big).
\end{equation*}
Since these subspaces all have different eigencharacters, any invariant subspace for the Heisenberg group $H_E$ of $E$ must be a sum of its intersections with these subspaces, which (as we have seen in \S~\ref{sec:proof-groups-with}) are permuted by $H_E$.
But each such subspace, as a representation of $H_{\R^n}$ is equivalent to $L^2(\R^n)$, so there are no non-trivial proper invariant subspaces.

Finally, if $\Hilb$ is a Hilbert space and $\rho$ is as in the second part of the statement of the Stone-von Neumann-Mackey Theorem, then restricting attention to the action of $H_E$,
\begin{equation*}
  \Hilb = \oplus_{\eta,a} \Hilb_{\eta,a}
\end{equation*}
as in \S~\ref{sec:proof-groups-with}, and each $\Hilb_{\eta,a}$ is invariant under $H_{\R^n}$ (since it commutes with $H_E$).
By the Stone-von Neumann theorem (real case),
\begin{equation*}
  \Hilb_{0,0}=\oplus_\alpha H^\alpha_{0,0},
\end{equation*}
in such a way that there exists an isometry $W^\alpha_{0,0}:L^2(\R^n)\to H^\alpha_{0,0}$ intertwining the actions of the Heisenberg group $H_{\R^n}$.
Let $\Hilb^\alpha$ be the closed linear span of the spaces $T_x M_\chi\Hilb^\alpha_{0,0}$ as $x$ and $\chi$ range over $E$ and $\hat E$.
Imitating the method of \S~\ref{sec:proof-groups-with} one may see that $\Hilb=\oplus_\alpha\Hilb^\alpha$.
Furthermore, if $v_\alpha=W^\alpha_{0,0}(\phi)$ (where $\phi\in L^2(\R^n)$ is the normalized Gaussian), then $W_\alpha:L^2(E\times \R^n)\to \Hilb^\alpha$ defined by $W_\alpha(T_x M_\chi(\delta_G\otimes \phi))=\rho(T_x M_\chi)v_\alpha$ for all $(x,\chi)\in L\times \hat L$ can be shown to be the necessary intertwining isometries, completing the proof of the Stone-von Neumann-Mackey Theorem.
\subsection*{Remark}
The proof given by von~Neumann in the real case carries over to the general case when the normalized Gaussian $\phi$ is replaced by $\delta_G\otimes \phi$.
Thus, the Stone-von Neumann-Mackey Theorem follows from the Fourier Inversion Theorem and Plancherel Theorem for locally compact abelian groups.
\subsection*{Acknowledgements}
We thank Tejas Kalelkar, Vipul Naik and Yashonidhi Pandey for their comments on a draft of this manuscript.

\end{document}